\documentclass{elsarticle}
\usepackage[latin9]{inputenc}
\usepackage{geometry}
\usepackage{mathrsfs}
\usepackage{mathtools}
\usepackage{algorithm2e}
\usepackage{amsthm}
\usepackage{amsmath}
\usepackage{amssymb}
\usepackage{mathdots}
\usepackage{stackrel}
\usepackage{esint}
\usepackage[unicode=true,
 bookmarks=false,
 breaklinks=false,pdfborder={0 0 1},backref=false,colorlinks=false]
 {hyperref}

\makeatletter
\numberwithin{figure}{section}
\theoremstyle{plain}
\newtheorem{thm}{\protect\theoremname}
  \theoremstyle{plain}
  \newtheorem{prop}[thm]{\protect\propositionname}
  \theoremstyle{plain}
  \newtheorem{lem}[thm]{\protect\lemmaname}
  \theoremstyle{remark}
  \newtheorem{rem}[thm]{\protect\remarkname}

\usepackage{amsthm}

\usepackage{lineno}
\modulolinenumbers[5]
\theoremstyle{plain}
\newtheorem*{criteria}{Criterion}

\makeatother

  \providecommand{\lemmaname}{Lemma}
  \providecommand{\propositionname}{Proposition}
  \providecommand{\remarkname}{Remark}
\providecommand{\theoremname}{Theorem}

\begin{document}

\begin{frontmatter}{}

\title{Criteria for linearized stability for a size-structured population
model}

\author{Inom Mirzaev}

\ead{mirzaev@colorado.edu}

\author{David M. Bortz}

\ead{dmbortz@colorado.edu}

\address{Department of Applied Mathematics, University of Colorado, Boulder,
CO, United States}
\begin{abstract}
We consider a size-structured aggregation and growth model of phytoplankton
community proposed by Ackleh and Fitzpatrick \cite{AcklehFitzpatrick1997}.
The model accounts for basic biological phenomena in phytoplankton
community such as growth, gravitational sedimentation, predation by
zooplankton, fecundity, and aggregation. Our primary goal in this
paper is to investigate the long-term behavior of the proposed aggregation
and growth model. Particularly, using the well-known principle of
linearized stability and semigroup compactness arguments, we provide
sufficient conditions for local exponential asymptotic stability of
zero solution as well as sufficient conditions for instability. We
express these conditions in the form of an easy to compute characteristic
function, which depends on the functional relationship between growth,
sedimentation and fecundity. Our results can be used to predict long-term
phytoplankton dynamics\end{abstract}
\begin{keyword}
\texttt{Nonlinear evolution equations, principle of linearized stability,
spectral analysis, structured populations dynamics, semigroup theory}
\end{keyword}

\end{frontmatter}{}

\section{Introduction}

Planktonic lifeforms provide a crucial source of food (organic carbon)
for many aquatic species including blue whales \cite{Reynolds1984}.
In fact, oceanic plankton collectively provide approximately $40\%$
of worlds organic carbon \cite{Falkowski1994}. The main component
of plankton community are unicellular algae called phytoplankton \cite{Lalli}.
Similar to terrestrial plants, phytoplankton make their living by
photosynthesis, and consequently inhabit surfaces of lakes and oceans.
Phytoplankton cannot swim against a current and thus form aggregated
communities at the surfaces of lakes and oceans to promote survival
and proliferation. Besides their macroscopic predators such as whales
and shrimp, aggregated phytoplankton community are also removed due
to gravitational sedimentation and zooplankton grazing on phytoplankton. 

To study the dynamic nature of phytoplankton communities various mathematical
population models have been developed. The mathematical model that
we consider in this article is the model first considered by Ackleh
and Fitzpatrick in \cite{AcklehFitzpatrick1997}. Although, the existence
and uniqueness of a global positive solution in several different
spaces have been proved in \cite{AcklehFitzpatrick1997,Ackleh1997NATMA,Lamb2009},
the long-term behavior of this model has not been investigated. This
is mainly due to the nonlinear nature of the Smoluchowski coagulation
equations used for modeling aggregation. Hence, our main goal in this
paper is to derive sufficient conditions for the local exponential
asymptotic \emph{stability} of the zero solution for the aggregation-growth
model (see Section \ref{sec:Linearized-stability-for}). In Section
\ref{sec:Linearized-instability-for}, we also derive sufficient conditions
for \emph{instability} of the zero solution. These conditions can
then be used to predict long-term phytoplankton dynamics. 

The size-structured population model proposed in \cite{AcklehFitzpatrick1997}
models growth of aggregates due to cell division and aggregation and
removal of aggregates due to sedimentation and microscopic predation
by zooplankton. In a phytoplankton community, the density of aggregates
of size $x$ at time $t$ is denoted by $p(t,\, x)$. An aggregate
is assumed to have minimum $x_{0}$ and maximum $x_{1}$ possible
sizes. In vivo, there are no aggregates of volume $0$ and the aggregates
cannot grow indefinitely, so the only biologically plausible case
is $0<x_{0}<x_{1}<\infty$. Hence, in this paper we consider the case
$0<x_{0}<x_{1}<\infty$, and postpone the analysis of the case with
$x_{0}=0$ and/or $x_{1}=\infty$ for our future papers. 

We will consider the following nonlinear partial integro-differential
equation model for the evolution of a phytoplankton population, 
\begin{equation}
\partial_{t}p(t,\, x)=-\partial_{x}(gp)+\mathcal{F}[p],\quad g(x_{0})p(t,\, x_{0})=\mathcal{K}[p](t),\quad p(0,\, x)=p_{0}(x)\in L^{1}[x_{0},\, x_{1}]\label{eq: agg and growth model}
\end{equation}
where 
\[
\mathcal{F}[p](t,\, x)=\frac{1}{2}\int_{x_{0}}^{x-x_{0}}\beta(x-y,\, y)p(t,\, x-y)p(t,\, y)\, dy-p(t,\, x)\int_{x_{0}}^{x_{1}}\beta(x,\, y)p(t,\, y)\, dy-w(x)p(t,\, x)
\]
and 
\begin{equation}
\mathcal{K}[p](t)=\int_{x_{0}}^{x_{1}}q(x)p(t,\, x)dx\,.\label{eq: entrance to single cell pop}
\end{equation}
The function $g(x)$ represents the average growth rate of the aggregate
of size $x$ due to mitosis. Specifically, when a single cell in the
aggregate of size $x$ divides into two identical parent and daughter
cells, the daughter cell enters the aggregate of size $x$ contributing
in a increase in total size. The coefficient $w(x)$ represents a
size-dependent removal rate. Biologically, aggregates can be either
removed by gravitational sedimentation or zooplankton grazing on phytoplankton
\cite{Griffin2001}. $\beta(x,\, y)$ is the aggregation kernel, which
describes the rate with which the aggregates of size $x$ and $y$
agglomerate to form an aggregate of size $x+y$. The fecundity rate
$q(x)$ in \eqref{eq: entrance to single cell pop} represents the
number of new cells that fall off an aggregate of size $x$ and enter
single cell population.

As our solution space we use $H=L^{1}[x_{0},\, x_{1}]$ with the usual
norm $\left\Vert \cdot\right\Vert _{L^{1}}$ (hereafter, just $\left\Vert \cdot\right\Vert $).
Consequently, the equation \eqref{eq: agg and growth model} can be
written as a semilinear abstract Cauchy problem (ACP)
\begin{equation}
p_{t}=\mathcal{L}[p]+\mathcal{N}[p],\qquad p(0,\, x)=b_{0}(x)\in H\,.\label{eq:Semilinear problem}
\end{equation}
 The operator $\mathcal{L}\,:\,\mathcal{D}(\mathcal{L})\subset H\to H$
is defined as 
\begin{equation}
\mathcal{L}[p](x)=-\left(g(x)p(x)\right)'-w(x)p(x)\label{eq:linear part}
\end{equation}
 with its corresponding domain
\begin{equation}
\mathcal{D}(\mathcal{L})=\left\{ \phi\in H\,|\,(g\phi)'\in H,\,(g\phi)(x_{0})=\mathcal{K}[\phi]\right\} \,.\label{eq:Second domain}
\end{equation}
The nonlinear operator $\mathcal{N}\,:\, H\to H$ is defined as 
\begin{equation}
\mathcal{N}[p]=\frac{1}{2}\int_{x_{0}}^{x-x_{0}}\beta(x-y,\, y)p(x-y)p(y)\, dy-p(x)\int_{x_{0}}^{x_{1}}\beta(x,\, y)p(y)\, dy\,.\label{eq:nonlinear part}
\end{equation}
We make the following assumptions
\begin{alignat*}{1}
 & g\in C^{1}[x_{0},\, x_{1}];\qquad\text{and }\, g(x)>0\:\text{ for }x_{0}\le x\le x_{1},\tag{\ensuremath{\mathscr{A}_{g}}}\\
 & \beta\in L^{\infty}([x_{0},\, x_{1}]\times[x_{0},\, x_{1}]);\qquad\beta(x,\, y)=\beta(y,\, x)\text{ and }\beta(x,\, y)=0\,\,\,\text{ if }x+y>x_{1}\,\tag{\ensuremath{\mathscr{A}_{\beta}}}\\
 & w\in C[x_{0},\, x_{1}]\qquad\text{and }w\ge0\text{ a.e. on }[x_{0},\, x_{1}]\,,\tag{\ensuremath{\mathscr{A}_{w}}}\\
 & q\in L^{\infty}[x_{0},\, x_{1}]\qquad\text{and }q\ge0\text{ a.e. on }[x_{0},\, x_{1}]\,.\tag{\ensuremath{\mathscr{A}_{q}}}
\end{alignat*}
Note that the restriction on $g$ states that any aggregate of size
$x\in[x_{0},\, x_{1}]$ has strictly positive growth rate. Any aggregate
growing out of the bounds are not considered in the model. The assumption
on $w(x)$ enforces continuous dependence of the removal on the size
of an aggregate and ensures that any aggregate size has a non-negative
removal rate. Having the required ingredients in hand, we present
the main result of this paper below, and demonstrate our proof in
the subsequent sections.
\begin{thm}
Under the assumptions ($\mathscr{A}_{g}$), ($\mathscr{A}_{\beta}$),
($\mathscr{A}_{w}$) and ($\mathscr{A}_{q}$), the zero solution of
the nonlinear evolution equation defined in \eqref{eq: agg and growth model}
is locally asymptotically stable if 
\[
\int_{x_{0}}^{x_{1}}\frac{q(x)}{g(x)}\exp\left(-\int_{x_{0}}^{x}\frac{w(s)}{g(s)}\, ds\right)\, dx<1\,.
\]
Moreover, the zero solution is unstable if 
\[
\int_{x_{0}}^{x_{1}}\frac{q(x)}{g(x)}\exp\left(-\int_{x_{0}}^{x}\frac{w(s)}{g(s)}\, ds\right)\, dx>1\,.
\]

\end{thm}

\section{\label{sec:Linearized-stability-for}Linearized stability for the
zero solution}

Our main goal in this section is to give sufficient conditions for
the zero solution, $p(t,\, x)\equiv0$, to be locally asymptotically
stable. In particular, we use the linearized stability results introduced
in \cite{Webb1985}. For the convenience of readers we now present
several results, which can be found in most of the semigroup theory
books (see \cite{Engel2000} for reference). 

The growth bound, $\omega_{0}(\mathcal{A})$, of a strongly continuous
semigroup $\left(T(t)\right)_{t\ge0}$ with an infinitesimal generator
$\mathcal{A}$ is defined as
\[
\omega_{0}(\mathcal{A}):=\inf\left\{ \omega\in\mathbb{R}\,:\,\begin{array}{c}
\exists M_{\omega}\ge1\text{ such that}\\
\left\Vert T(t)\right\Vert \le M_{\omega}e^{\omega t}\text{ for all }t\ge0
\end{array}\right\} \,.
\]
In the development below we denote $D\mathcal{A}(f)$ to be the Fréchet
derivative of an operator $\mathcal{A}$ evaluated at $f$, which
is defined as 
\[
D\mathcal{A}(u)h=\mathcal{A}[u+h]-\mathcal{A}[u]+o(h),\qquad\forall u\in\mathcal{D}(\mathcal{A})\,,
\]
where $o$ is little-o operator satisfying $\left\Vert o(h)\right\Vert \le b(r)\left\Vert h\right\Vert $
with increasing continuous function\textbf{ $b\,:\,[0,\,\infty)\to[0,\,\infty),\, b(0)=0$}. 

The stability results of this section are based on the following proposition
from \cite[p.198]{Webb1985}. We refer readers to \cite{Kato1995}
for a generalized version of this proposition, which applies to a
broader range of nonlinear evolution equations.
\begin{prop}
\label{prop:linearized stability}Let $\left(T(t)\right)_{t\ge0}$
be a $C_{0}$ semigroup in the Banach space $X$ with an infinitesimal
generator $\mathcal{L}$. Let $\mathcal{N}\,:\, X\to X$ be continuously
Fréchet differentiable on $X$. Let $\bar{f}\in\mathcal{D}(\mathcal{L})$
be a stationary solution of \eqref{eq:Semilinear problem}, i.e.,
$\left(\mathcal{L}+\mathcal{N}\right)[\bar{f}]=0$. If the linearized
operator $\mathcal{L}+D\mathcal{N}(\bar{f})$ (which is the infinitesimal
generator of a $C_{0}$ semigroup by the well-known perturbation theorem)
satisfies $\omega_{0}\left(\mathcal{L}+D\mathcal{N}(\bar{f})\right)<0$,
then $\bar{f}$ is locally asymptotically stable in the following
sense: 

There exists $\eta,\, C\ge1,$ and $\alpha>0$ such that if $\left\Vert f-\bar{f}\right\Vert <\eta$,
then a unique mild solution $u(t)$ of \eqref{eq:Semilinear problem},
\[
u(t)=T(t)f+\int_{0}^{t}T(t-s)\mathcal{N}[u(s)]\, ds\,,
\]
exists for all $t\ge0$ and $\left\Vert u(t)-\bar{f}\right\Vert \le Ce^{-\alpha t}\left\Vert f-\bar{f}\right\Vert $
for all $t\ge0$. 
\end{prop}
In \cite{Lamb2009} authors have proved that the linear operator $\mathcal{L}$
generates a strongly continuous semigroup. Hence, we prove the second
assumption of Proposition \ref{prop:linearized stability} below.
\begin{lem}
The nonlinear operator $\mathcal{N}$ defined in \eqref{eq:nonlinear part}
is continuously Fréchet differentiable on $H$.\end{lem}
\begin{proof}
The Fréchet derivative of the nonlinear operator $\mathcal{N}$ is
given explicitly as

\[
D\mathcal{N}(\phi)[h(x)]=\frac{1}{2}\int_{x_{0}}^{x-x_{0}}\beta(x-y,\, y)\left[\phi(y)h(x-y)+h(y)\phi(x-y)\right]dy-h(x)\int_{x_{0}}^{x_{1}}\beta(x,\, y)\phi(y)dy-\phi(x)\int_{x_{0}}^{x_{1}}\beta(x,\, y)h(y)dy\,.
\]
For the arbitrary functions $u_{1},\, u_{2}\in H$ we have 
\begin{alignat*}{1}
\left|D\mathcal{N}(u_{1})h(x)-D\mathcal{N}(u_{2})h(x)\right| & \le\frac{1}{2}\left\Vert \beta\right\Vert _{\infty}\int_{x_{0}}^{x_{1}}|u_{1}(y)-u_{2}(y)||h(x-y)|\, dy+\frac{1}{2}\left\Vert \beta\right\Vert _{\infty}\int_{x_{0}}^{x_{1}}|h(y)||u_{1}(x-y)-u_{2}(x-y)|\, dy\\
 & +|h(x)|\left\Vert \beta\right\Vert _{\infty}\int_{x_{0}}^{x_{1}}|u_{1}(y)-u_{2}(y)|dy+|u_{1}(x)-u_{2}(x)|\left\Vert \beta\right\Vert _{\infty}\int_{x_{0}}^{x_{1}}|h(y)|dy
\end{alignat*}
An application of Young's inequality for convolutions (see \cite[Theorem 2.24]{Adams2003})
to the first two integrals gives,
\begin{alignat*}{1}
\left|D\mathcal{N}(u_{1})h(x)-D\mathcal{N}(u_{2})h(x)\right| & \le\left\Vert \beta\right\Vert _{\infty}\left\Vert u_{1}-u_{2}\right\Vert \left\Vert h\right\Vert +|h(x)|\left\Vert \beta\right\Vert _{\infty}\left\Vert u_{1}-u_{2}\right\Vert +|u_{1}(x)-u_{2}(x)|\left\Vert \beta\right\Vert _{\infty}\left\Vert h\right\Vert \,.
\end{alignat*}
Consequently, taking the integral of both sides with respect to $x$
yields
\[
\left\Vert D\mathcal{N}(u_{1})h(x)-D\mathcal{N}(u_{2})h(x)\right\Vert \le(x_{1}-x_{0})\left\Vert \beta\right\Vert _{\infty}\left\Vert u_{1}-u_{2}\right\Vert \left\Vert h\right\Vert +\left\Vert \beta\right\Vert _{\infty}\left\Vert u_{1}-u_{2}\right\Vert \left\Vert h\right\Vert +\left\Vert u_{1}-u_{2}\right\Vert \left\Vert \beta\right\Vert _{\infty}\left\Vert h\right\Vert 
\]
for all $h\in H$. Then it follows that 
\[
\left\Vert D\mathcal{N}(u_{1})-D\mathcal{N}(u_{2})\right\Vert \le(x_{1}-x_{0}+2)\left\Vert \beta\right\Vert _{\infty}\left\Vert u_{1}-u_{2}\right\Vert \,,
\]
which in turn implies that the nonlinear operator $\mathcal{N}$ is
continuously Fréchet differentiable on $H$.
\end{proof}
For the zero solution, $\bar{f}=0$, we have $D\mathcal{N}(\bar{f})=0$.
Proposition \ref{prop:linearized stability} implies that in order
to derive a sufficient condition, we also have to show that $\omega_{0}(\mathcal{L})<0$.
To achieve that we follow the steps introduced in \cite[\S 6.1]{Engel2000}
and \cite{Greiner1988a}. First, we show that $C_{0}$ semigroup $\left(T(t)\right)_{t\ge0}$
generated by $\mathcal{L}$ is, in fact, a positive semigroup on the
\emph{Banach lattice} $H$ with the usual sense of ordering ``\textbf{$\ge$}''
and the absolute value, $|\cdot|$.%
\footnote{We refer readers to \cite{Nagel1986} for a detailed definition of
a positive semigroup on a Banach lattice.%
}
\begin{prop}
\label{prop:growth bound equals spectral bound}The linear operator
$\mathcal{L}$ on $H=L^{1}[x_{0},\, x_{1}]$ generates a positive
$C_{0}$ semigroup $\left(T(t)\right)_{t\ge0}$. Hence, the growth
bound $\omega_{0}(\mathcal{L})$ is equal to the spectral bound $s(\mathcal{L})$,
i.e., $s(\mathcal{L})=\omega_{0}(\mathcal{L})\,.$\end{prop}
\begin{proof}
For given $\lambda\in\mathbb{C}$, the resolvent operator of $\mathcal{L}$
is given explicitly by 

\begin{equation}
R(\lambda,\,\mathcal{L})\phi=\phi(x_{0})g(x_{0})\frac{q(x)}{g(x)}\exp\left(-\int_{x_{0}}^{x}\frac{\lambda+w(s)}{g(s)}\, ds\right)+\frac{1}{g(x)}\int_{x_{0}}^{x}\phi(y)\exp\left(-\int_{y}^{x}\frac{\lambda+w(s)}{g(s)}\, ds\right)dy\label{eq:resolvent operator}
\end{equation}
Using the similar reasoning to Lemma 2.4 of \cite{Lamb2009} we conclude
that for all $\lambda\in\mathbb{C}$ with $Re(\lambda)>\left\Vert g\right\Vert _{\infty}+\left\Vert q\right\Vert _{\infty}$,
the resolvent operator $R(\lambda,\,\mathcal{L})$ exists. Thus, the
resolvent set of $\mathcal{L}$ is not empty, i.e., $\rho(\mathcal{L})\not\equiv\emptyset\,.$
Then, since $g$, $w$, and $q$ are positive functions on $[x_{0},\, x_{1}]$,
it is straightforward to see that $R(\lambda,\,\mathcal{L})$ is a
positive operator whenever it exists. Consequently, the positivity
of $T(t)$ follows from the Characterization Theorem in \cite[p.353]{Engel2000}.
Since $H=L^{1}[x_{0},\, x_{1}]$ is a Banach lattice and $\left(T(t)\right)_{t\ge0}$
is a positive semigroup, the last statement of the proposition follows
from the main theorem presented in \cite{Weis1995}. 
\end{proof}

Recall that an operator $\mathcal{A}$ has a \emph{compact resolvent}
if $\rho(\mathcal{A})\ne\emptyset$ and $R(\lambda,\,\mathcal{A})$
is \emph{compact} for all $\lambda\in\rho(\mathcal{A})$. Moreover,
if one can prove that $R(\lambda,\,\mathcal{A})$ is compact for some
$\lambda\in\rho(\mathcal{A})$, then $R(\lambda,\,\mathcal{A})$ is
compact for all $\lambda\in\rho(\mathcal{A})$ . Operators which have
a compact resolvent have the very nice property that their spectrum
$\sigma(\mathcal{A})$ coincide with their point spectrum $\sigma_{p}(\mathcal{A})$,
i.e., 
\begin{equation}
\sigma(\mathcal{A})=\sigma_{p}(\mathcal{A}):=\left\{ \lambda\in\mathbb{C}\,|\,\mathcal{A}\phi=\lambda\phi\text{ for some }\phi\ne0\in X\right\} \,,\label{eq:coincides with point spectrum}
\end{equation}
where $X$ is some Banach space and $\mathcal{A}$ is an operator
on $X$ with a compact resolvent. 

Before proceeding further we first prove the following lemma, which
we use in the proof of the next proposition. 
\begin{lem}
\label{lem:compactness}For $\phi\in H$ and $\lambda_{1}\in\mathbb{C}$
with $Re(\lambda_{1})>\left\Vert g\right\Vert _{\infty}+\left\Vert q\right\Vert _{\infty}$
we define the linear operators
\[
K_{1}[\phi](x)=\phi(x_{0})g(x_{0})=\int_{x_{0}}^{x_{1}}q(y)\phi(y)\, dy
\]
and 
\[
K_{2}[\phi](x)=\int_{x_{0}}^{x}\phi(y)\exp\left(-\int_{y}^{x}\frac{\lambda_{1}+w(s)}{g(s)}\, ds\right)dy\,.
\]
Then $K_{1}$ and $K_{2}$ are compact on $H$.\end{lem}
\begin{proof}
We first prove that $K_{2}$ is compact on $H$. Recall that an operator
is compact if it maps bounded sets into relatively compact sets. In
order to show that a set is relatively compact in $L^{1}$, we use
the Kolmogorov-Riesz compactness theorem (see \emph{\cite{Hanche-Olsen2010}}
for a statement of the theorem). 

Define a unit ball in $H=L^{1}[x_{0},\, x_{1}]$ as $B=\left\{ \phi\in H\,|\,\left\Vert \phi\right\Vert \le1\right\} $.
Now, we prove that $K_{2}B$ is relatively compact by proving that
each condition of the Kolmogorov-Riesz theorem applies to the set
$K_{2}B$. For a given function $\phi\in B$,
\begin{alignat*}{1}
\left\Vert K_{2}\phi(x)\right\Vert  & =\int_{x_{0}}^{x_{1}}\left|\int_{x_{0}}^{x}\phi(y)\exp\left(-\int_{y}^{x}\frac{\lambda_{1}+w(s)}{g(s)}\, ds\right)\, dy\right|\, dx\le\int_{x_{0}}^{x_{1}}\int_{x_{0}}^{x}|\phi(y)|\, dy\, dx\\
 & \le(x_{1}-x_{0})\left\Vert \phi\right\Vert \le(x_{1}-x_{0})\,.
\end{alignat*}
This implies that the set $K_{2}B$ is bounded. So the first condition
of the Kolmogorov-Riesz theorem is satisfied. Since we are working
on a finite domain $[x_{0},\, x_{1}]$, the second condition of the
Kolmogorov-Riesz theorem is satisfied by default. 

For convenience, let us define 
\[
k(x,\, y):=\exp\left(-\int_{y}^{x}\frac{\lambda_{1}+w(s)}{g(s)}\, ds\right).
\]
Observe that for $\lambda_{1}\in\mathbb{C}$ with $Re(\lambda_{1})>\left\Vert g\right\Vert _{\infty}+\left\Vert q\right\Vert _{\infty}$,
$k(x,\, y)$ is uniformly continuous on $[x_{0},\, x_{1}]\times[x_{0},\, x_{1}]$.
So for a given $\varepsilon_{1}>0$ there exist $\delta_{1}>0$ such
that
\begin{equation}
\left|k(x+h,\, y)-k(x,\, y)\right|<\varepsilon_{1}\text{ for all }|h|<\delta_{1}\text{ and }y\in[x_{0},\, x_{1}]\,.\label{eq:uniform cont}
\end{equation}
Furthermore, notice that $|k(x,\, y)|\le1$ for $x\ge y$. Consequently,
for a given $\varepsilon>0$ and $\phi\in B$ we have 
\begin{alignat*}{1}
\int_{x_{0}}^{x_{1}}\left|\int_{x_{0}}^{x+h}k(x+h,\, y)\phi(y)\, dy-\int_{x_{0}}^{x}k(x,\, y)\phi(y)\, dy\right|\, dx & \le\int_{x_{0}}^{x_{1}}\left|\int_{x}^{x+h}k(x+h,\, y)\phi(y)\, dy\right|\, dx\\
 & +\int_{x_{0}}^{x_{1}}\left|\int_{x_{0}}^{x}\left[k(x+h,\, y)-k(x,\, y)\right]\phi(y)\, dy\right|\, dx\\
 & \le\int_{x_{0}}^{x_{1}}\left|\int_{x}^{x+h}\left|k(x+h,\, y)\right|\cdot\left\Vert \phi\right\Vert \, dy\right|\, dx+\int_{x_{0}}^{x_{1}}\int_{x_{0}}^{x}\varepsilon_{1}|\phi(y)|\, dy\, dx\\
 & \le(x_{1}-x_{0})|h|\cdot\left\Vert \phi\right\Vert +\varepsilon_{1}(x_{1}-x_{0})\cdot\left\Vert \phi\right\Vert \\
 & \le(\delta_{1}+\varepsilon_{1})(x_{1}-x_{0})<\varepsilon
\end{alignat*}
for sufficiently small $\varepsilon_{1}$ and $\delta_{1}$ in \eqref{eq:uniform cont}.
This proves the third condition, and thus from the Kolmogorov-Riesz
theorem the set $K_{2}B$ is relatively compact in $H$. This in turn
implies that $K_{2}$ is a compact operator on $H$.

For the operator $K_{1}$ observe that the set $K_{1}B$ is bounded,
i.e.,
\[
\left\Vert K_{1}\phi(x)\right\Vert =\int_{x_{0}}^{x_{1}}\left|\int_{x_{0}}^{x_{1}}\phi(y)q(y)\, dy\right|\, dx\le(x_{1}-x_{0})\left\Vert q\right\Vert _{\infty}\left\Vert \phi\right\Vert \le(x_{1}-x_{0})\left\Vert q\right\Vert _{\infty}\,.
\]
Since the function $q(y)$, inside the integral, does not depend on
$x$, the third condition of the Kolmogorov-Riesz theorem follows
immediately. Hence, the operator $K_{1}$ is also compact.\end{proof}
\begin{prop}
\label{prop:Compact resolvent}The operator $\mathcal{L}$ defined
in \eqref{eq:linear part} has a compact resolvent. \end{prop}
\begin{proof}
From Proposition \ref{prop:growth bound equals spectral bound} we
already know that the resolvent set of $\mathcal{L}$ is not empty.
So we only have to prove that $R(\lambda,\,\mathcal{L})$ is compact
for some $\lambda\in\rho(\mathcal{L})$. Particularly, we will prove
that $R(\lambda_{1},\,\mathcal{L})$ is compact for any $\lambda_{1}\in\mathbb{C}$
with $Re(\lambda_{1})>\left\Vert g\right\Vert _{\infty}+\left\Vert q\right\Vert _{\infty}$.
The resolvent operator defined in \eqref{eq:resolvent operator} can
be written as the sum of the compositions of linear operators 
\[
R(\lambda_{1},\,\mathcal{L})=B_{1}K_{1}+B_{2}K_{2}\,,
\]
where $K_{1}$ and $K_{2}$ are defined as in Lemma \ref{lem:compactness},
\begin{alignat*}{1}
B_{1}[\phi](x) & =\frac{q(x)}{g(x)}\exp\left(-\int_{x_{0}}^{x}\frac{\lambda_{1}+w(s)}{g(s)}\, ds\right)\phi(x)\,,
\end{alignat*}
and
\[
B_{2}[\phi](x)=\frac{1}{g(x)}\phi(x)\,.
\]
From the assumptions ($\mathscr{A}_{q}$) and ($\mathscr{A}_{g}$)
it follows that the operators $B_{1}$ and $B_{2}$ are bounded. In
Lemma \ref{lem:compactness} we have proved that the operators $K_{1}$
and $K_{2}$ are compact on $H$. Then the operator $B_{1}K_{1}$,
composition of a bounded and a compact operator, is compact (see Proposition
5.43 of \cite{Hunter2000}). Similarly, the operator $B_{2}K_{2}$
is compact. This in turn implies that the resolvent operator $R(\lambda_{1},\,\mathcal{L})$,
a linear combination of compact operators, is also compact. 
\end{proof}

The above proposition together with equation \eqref{eq:coincides with point spectrum}
implies that the spectrum of the operator $\mathcal{L}$ consists
of only eigenvalues, i.e., $\sigma(\mathcal{L})=\sigma_{p}(\mathcal{L})$.
Thus, we can now characterize the spectrum of $\mathcal{L}$ by its
eigenvalues. 
\begin{prop}
\label{prop:existence condition}For $\lambda\in\mathbb{C}$, 
\[
\lambda\in\sigma_{p}(\mathcal{L})=\sigma(\mathcal{L})\,\Leftrightarrow\,\xi(\lambda)=0\,,
\]
where 
\[
\xi(\lambda)=\int_{x_{0}}^{x_{1}}\frac{q(x)}{g(x)}\exp\left(-\int_{x_{0}}^{x}\frac{\lambda+w(s)}{g(s)}\, ds\right)\, dx-1
\]
is a characteristic function of $\mathcal{L}$. \end{prop}
\begin{proof}
The eigenvalue equation for the operator $\mathcal{L}$ is given by
\[
\mathcal{L}\phi-\lambda\phi=-\left(g(x)\phi(x)\right)'-w(x)\phi(x)-\lambda\phi=0\,.
\]
The solution of the above equation is given by the following eigenfunction
\begin{equation}
\phi(x)=\frac{\phi(x_{0})g(x_{0})}{g(x)}\exp\left(-\int_{x_{0}}^{x}\frac{\lambda+w(s)}{g(s)}\, ds\right)\,.\label{eq:eigenvector}
\end{equation}
We note that 
\begin{alignat*}{1}
\left\Vert \phi\right\Vert  & =\int_{x_{0}}^{x_{1}}\frac{|\phi(x_{0})g(x_{0})|}{g(x)}\exp\left(-Re(\lambda)\int_{x_{0}}^{x}\frac{1}{g(s)}\, ds\right)\exp\left(-\int_{x_{0}}^{x}\frac{w(s)}{g(s)}\, ds\right)\, dx\\
 & \le\int_{x_{0}}^{x_{1}}\frac{|\phi(x_{0})g(x_{0})|}{g(x)}\exp\left(-Re(\lambda)\int_{x_{0}}^{x}\frac{1}{g(s)}\, ds\right)\, dx\\
 & \le(x_{1}-x_{0})|\phi(x_{0})g(x_{0})|\left\Vert \frac{1}{g(x)}\right\Vert _{\infty}\exp\left(-Re(\lambda)\int_{x_{0}}^{x_{1}}\frac{1}{g(s)}\, ds\right)\,,
\end{alignat*}
and hence $\phi\in H$. From assumption ($\mathscr{A}_{g}$) we have
$g\in C^{1}[x_{0},\, x_{1}]$. This in turn implies that $g'\phi\in H$.
Analogously, we can also prove that $(g\phi)'=g'\phi+g\phi'\in H$.

Lastly, in order for $\phi\in\mathcal{D}(\mathcal{L})$ we should
have 
\[
g(x_{0})\phi(x_{0})=\mathcal{K}[\phi]=\int_{x_{0}}^{x_{1}}\phi(x_{0})g(x_{0})\frac{q(x)}{g(x)}\exp\left(-\int_{x_{0}}^{x}\frac{\lambda+w(s)}{g(s)}\, ds\right)\,,
\]
which is equivalent to 
\[
0=\xi(\lambda)=\int_{x_{0}}^{x_{1}}\frac{q(x)}{g(x)}\exp\left(-\int_{x_{0}}^{x}\frac{\lambda+w(s)}{g(s)}\, ds\right)\, dx-1
\]

\end{proof}
Since the eigenfunction $\phi$ defined in \eqref{eq:eigenvector}
is $L^{1}$, the characteristic function $\xi(\lambda)$ is a finite-valued
function for all $\lambda\in\mathbb{C}$. Moreover, when $\xi(\lambda)$
is restricted to $\mathbb{R}$ it is strictly decreasing. Furthermore,
a simple limit calculation shows that 
\[
\lim_{\lambda\to\infty}\xi(\lambda)=-1\text{ and }\lim_{\lambda\to-\infty}\xi(\lambda)=\infty\,.
\]
This in turn, from the Intermediate Value Theorem, implies that there
exists a unique $\lambda_{0}\in\mathbb{R}$ such that $\xi(\lambda_{0})=0$.
We can also guarantee that this eigenvalue $\lambda_{0}$ is negative
real number provided that we have $\xi(0)<0$.
\begin{rem}
\label{rem:We-also-claim}We also claim that the spectral bound of
$\mathcal{L}$ is equal to this $\lambda_{0}$, i.e., $s(\mathcal{L})=\lambda_{0}$.
Suppose that there exists $\lambda_{1}\in\sigma(\mathcal{L})=\sigma_{p}(\mathcal{L})\subseteq\mathbb{C}$
such that $Re(\lambda_{1})>\lambda_{0}$. On the other hand, from
Proposition \ref{prop:existence condition}, $\lambda_{1}$ should
be a zero of characteristic function. i.e. $\xi(\lambda_{1})=0$.
However, 
\begin{alignat*}{1}
1=Re(\xi(\lambda_{1}))+1 & =\int_{x_{0}}^{x_{1}}\frac{q(x)}{g(x)}\cos\left[\int_{x_{0}}^{x}\frac{Im(\lambda_{1})}{g(s)}\, ds\right]\exp\left(-\int_{x_{0}}^{x}\frac{Re(\lambda_{1})+w(s)}{g(s)}\, ds\right)\, dx\\
 & \le\int_{x_{0}}^{x_{1}}\frac{q(x)}{g(x)}\exp\left(-\int_{x_{0}}^{x}\frac{Re(\lambda_{1})+w(s)}{g(s)}\, ds\right)\, dx\\
 & <\int_{x_{0}}^{x_{1}}\frac{q(x)}{g(x)}\exp\left(-\int_{x_{0}}^{x}\frac{\lambda_{0}+w(s)}{g(s)}\, ds\right)\, dx\\
 & =1\,,
\end{alignat*}
which is a contradiction. 
\end{rem}
We summarize the discussion of this section in the following criterion. 

\begin{criteria}(Stability) The spectral bound of \emph{$\mathcal{L}$}
is a unique real number $\lambda_{0}$ such that $\xi(\lambda_{0})=0$
and hence $\omega_{0}(\mathcal{L})=s(\mathcal{L})=\lambda_{0}$. Moreover,
the zero solution of the semilinear evolution equation \eqref{eq:Semilinear problem}
is locally exponentially stable if 
\[
\xi(0)=\int_{x_{0}}^{x_{1}}\frac{q(x)}{g(x)}\exp\left(-\int_{x_{0}}^{x}\frac{w(s)}{g(s)}\, ds\right)-1<0\,.
\]
\end{criteria}

\section{\label{sec:Linearized-instability-for}Linearized instability for
the zero solution}

In this section we derive sufficient conditions for instability of
the zero solution. In particular, we use the following proposition
from \cite[p.206]{Webb1985}. 
\begin{prop}
\label{prop:Instability theorem} Let $\left(T(t)\right)_{t\ge0}$
be a $C_{0}$ semigroup in the Banach space $X$ with infinitesimal
generator $\mathcal{L}$. Let $\mathcal{N}\,:\, X\to X$ be continuously
Fréchet differentiable on $X$. Let $\bar{f}\in\mathcal{D}(\mathcal{L})$
be a stationary solution of \eqref{eq:Semilinear problem}. If there
exists $\lambda_{0}\in\sigma(\mathcal{L}+D\mathcal{N}(\bar{f}))$
such that $Re(\lambda_{0})>0$ and 
\begin{equation}
\max\left\{ \omega_{1}(\mathcal{L}+D\mathcal{N}(\bar{f})),\,\sup_{\lambda\in\sigma_{D}(\mathcal{L}+D\mathcal{N}(\bar{f}))\backslash\{\lambda_{0}\}}Re(\lambda)\right\} <Re(\lambda_{0})\,,\label{eq: instability condition}
\end{equation}
then $\bar{f}$ is an unstable equilibrium in the sense that there
exists $\varepsilon>0$ and sequence $\left\{ f_{n}\right\} $ in
$X$ such that $f_{n}\to\bar{f}$ and $\left\Vert T(n)f_{n}-\bar{f}\right\Vert \ge\varepsilon$
for $n=1,2,\dots$ .
\end{prop}

The\emph{ discrete spectrum} of an operator $\mathcal{A}$ on a Banach
space $X$, denoted by $\sigma_{D}(\mathcal{A})$, is the subset of
$\lambda\in\sigma_{p}(\mathcal{A})$ such that $\lambda$ is an isolated
eigenvalue of finite multiplicity, i.e., the dimension of the set
$\left\{ \psi\in X\,:\,\mathcal{A}\psi=\lambda\psi\right\} $ is finite
and nonzero. Let $\left(T(t)\right)_{t\ge0}$ be a $C_{0}$ semigroup
on the Banach space $X$ with its generator $\mathcal{A}$. Then the
limit $\omega_{1}(\mathcal{A})=\lim_{t\to\infty}t^{-1}\log\left(\alpha[T(t)]\right)$
is called $\alpha$-growth bound of $\left(T(t)\right)_{t\ge0}$.
Here, $\alpha[T(t)]$ is a measure of non-compactness of $T(t)$.
The measure of non-compactness, introduced in a textbook \cite{Kuratowski},
associates numbers to operators (or sets), which tells how close is
an operator (or a set) to a compact operator (or set). For instance,
$\alpha[T(t)]=0$ implies that the semigroup $\left(T(t)\right)_{t\ge0}$
is eventually compact. In general, computing an explicit value of
the $\alpha-$growth bound $\omega_{1}$ is a complicated task. However,
if we can prove that the linear operator $\mathcal{L}+D\mathcal{N}(\bar{f})$
generates an eventually compact $C_{0}$ semigroup, then from \cite[Remark 4.8]{Webb1985}
it follows that $\omega_{1}(\mathcal{L}+D\mathcal{N}(\bar{f}))=-\infty$. 
\begin{prop}
\label{Prop:eventually compact} The semigroup $T(t)$ generated by
the operator $\mathcal{L}$ is eventually compact. Specifically, it
is compact for $t>2\Gamma(x_{1})$, where 
\[
\Gamma(x):=\int_{x_{0}}^{x}\frac{1}{g(s)}\, ds\,.
\]
 \end{prop}
\begin{proof}
We first show that the semigroup $\left(T(t)\right)_{t\ge0}$ is differentiable
for $t>2\Gamma(x_{1})$. Note that eventual differentiability implies
eventual norm continuity (see Diagram 4.26 in \cite[p.119]{Engel2000}).
Consequently, the result follows from the compactness of the resolvent
set of $\mathcal{L}$ (Proposition \ref{prop:Compact resolvent})
and \cite[Lemma 4.28]{Engel2000}. 

The following proof has been adopted from the proof of \cite[Theorem 3.1]{Farkas2007}.
The abstract Cauchy problem $u_{t}=\mathcal{L}u$ can be rewritten
as a partial differential equation
\begin{equation}
u_{t}(t,\, x)+g(x)u_{x}(t,\, x)+\left(g'(x)+w(x)\right)u(t,\, x)=0\,.\label{eq: simplified ACP}
\end{equation}
By Thoerem 2.2 in \cite{BanksKappel1989} we know that the semigroup
$\left(T(t)\right)_{t\ge0}$ generated by \emph{$\mathcal{L}$} is
given explicitly by
\begin{equation}
T(t)\varphi(x)=\begin{cases}
\varphi\left(\Gamma^{-1}\left(\Gamma(x)-t\right)\right)\frac{g\left(\Gamma^{-1}\left(\Gamma(x)-t\right)\right)}{g(x)}\exp\left(-\int_{\Gamma^{-1}\left(\Gamma(x)-t\right)}^{x}\frac{w(s)}{g(s)}\, ds\right) & t\le\Gamma(x)\\
\frac{1}{g(x)}\mathcal{K}\left(T\left(t-\Gamma(x)\right)\varphi(x)\right)\exp\left(-\int_{x_{0}}^{x}\frac{w(s)}{g(s)}\, ds\right) & \Gamma(x)<t
\end{cases}\,.\label{eq:semigroup T(t)}
\end{equation}
 Thus for $t>\Gamma(x)$, we have 
\[
u(t,\, x)=\frac{1}{g(x)}\int_{x_{0}}^{x_{1}}q(x)u(t-\Gamma(x),\, x)dx\times\exp\left(-\int_{x_{0}}^{x}\frac{w(s)}{g(s)}\, ds\right)\,.
\]
On the other hand, since $g\in C^{1}[x_{0},\, x_{1}]$ and $g(x)>0$
we have $\Gamma(x)\in C[x_{0},x_{1}]$. Therefore, $u(t,\, x)$ is
continuous both in $x$ and $t$ for $t>\Gamma(x_{1})$, where $\Gamma(x_{1})$
is the maximum of $\Gamma(x)$ on $[x_{0},\, x_{1}]$. Then, from
\eqref{eq: simplified ACP}, $u(t,\, x)$ is continuously differentiable
for $t>2\Gamma(x_{1})$.
\end{proof}
In Proposition \ref{prop:Compact resolvent} we have proved that the
spectrum of $\mathcal{L}$ consists of only eigenvalues, which can
be expressed as the zeros of the characteristic equation $\xi(\lambda)$.
Consequently, similar to Remark \ref{rem:We-also-claim}, we can show
that there exist a unique real number $\lambda_{0}>0$ such that 
\[
\sup_{\lambda\in\sigma_{D}(\mathcal{L}+D\mathcal{N}(\bar{f}))\backslash\{\lambda_{0}\}}Re(\lambda)\le\sup_{\lambda\in\sigma_{p}(\mathcal{L}+D\mathcal{N}(\bar{f}))\backslash\{\lambda_{0}\}}Re(\lambda)<Re(\lambda_{0})
\]
if and only if
\[
\xi(0)=\int_{x_{0}}^{x_{1}}\frac{q(x)}{g(x)}\exp\left(-\int_{x_{0}}^{x}\frac{w(s)}{g(s)}\, ds\right)-1>0\,.
\]
 We can now summarize the results of this section in the following
criterion.

\begin{criteria}(Instability) The $\alpha-$growth bound of the operator
$\mathcal{L}$ can be found explicitly, i.e., $\omega_{1}(\mathcal{L})=-\infty$.
Furthermore, if 
\[
\xi(0)=\int_{x_{0}}^{x_{1}}\frac{q(x)}{g(x)}\exp\left(-\int_{x_{0}}^{x}\frac{w(s)}{g(s)}\, ds\right)-1>0\,,
\]
then the zero solution of the semilinear evolution equation \eqref{eq:Semilinear problem}
is unstable in the sense described in Proposition \ref{prop:Instability theorem}.\end{criteria}

\section{\label{sec:Discussion-and-future}Concluding remarks}

Expressing stability results in terms of characteristic functions
has been popular since A. Lotka published his pioneering article on
population modeling \cite{Lotka1922}. The characteristic function,
$\xi(\lambda)$, that we derived in this paper is easy to compute.
For example, consider a hypothetical scenario for which the model
parameters are given by $x_{0}=1,\, x_{1}=1000$, $q(x)=\ln(x)$,
$g(x)=\frac{1}{10}x(1001-x)$, $w(x)=\frac{1}{1000}(x-1)^{1.17}$.
In this case the zero solution of the evolution equation \eqref{eq: agg and growth model}
is locally asymptotically stable. On the other hand, decreasing the
growth rate twofold gives an unstable zero solution. This might not
be very intuitive at first glance. However, when we decrease the growth
rate twofold, aggregates will grow slower, but new cells keep entering
single cell population at the same rate (fecundity rate $q(x)$).
Thus, to keep the zero solution stable we should also decrease the
fecundity rate (at least) twofold. Furthermore, as one might expect,
doubling the fecundity rate $q(x)$ also makes the zero solution unstable. 

As a future research, we plan to extend the results of this paper
to nontrivial solutions (stationary, self-similar, etc.) of this aggregation-growth
model as well as apply this linearization method to a size-specific
aggregation-fragmentation model considered in \cite{Bortz2008}.

\section{Acknowledgements}

Funding for this research was supported in part by grants NIH-NIGMS
2R01GM069438-06A2 and NSF-DMS 1225878.

\section*{References}

\bibliographystyle{apalike}
\bibliography{library}

\end{document}